\documentclass[a4paper,12pt,reqno]{amsart}
\usepackage{amssymb}
\usepackage{amsmath}

\def\vs{\vskip .6cm}

\def\beq{\begin{equation}}
\def\eeq{\end{equation}}
\def\bea{\begin{eqnarray*}}
\def\eea{\end{eqnarray*}}
\def\f{\varphi}

\def\r{\end{proof}}

\def\dt{{\partial_t}}
\def\es{\,\lrcorner\,}
\def\id{\mathrm{id}}
\def\Lie{{\mathcal L}}
\def\SU{\rm{SU}}

\def \RM{\mathbb{R}}

\def \ZM{\mathbb{Z}}
\def \CM{\mathbb{C}}

\newtheorem{ede}{Definition}[section]

\newtheorem{ath}[ede]{Theorem}

\newtheorem{elem}[ede]{Lemma}

\newtheorem{ere}[ede]{Remark}

\newtheorem{ecor}[ede]{Corollary}

\title{Transformations of locally conformally K\"ahler manifolds}
\author{Andrei Moroianu}
\author{Liviu Ornea}
\thanks{This work was accomplished in the framework of the
  Associated
  European Laboratory ``MathMode". L.O. is partially supported by
CNCSIS PNII  grant code 8.}

\address{Centre de Math{\'e}mathiques, Ecole Polytechnique, 91128
  Palaiseau Cedex, France} 
\email{am@math.polytechnique.fr}
\address{Univ. of Bucharest, Faculty of Mathematics,
14 Academiei str, 70109 Bucharest, Romania, and 
Institute of Mathematics ``Simion Stoilow" of the Romanian Academy,
21, Calea Grivitei str., 
010702-Bucharest, Romania.}
\email{lornea@gta.math.unibuc.ro, Liviu.Ornea@imar.ro}

\begin{document}

\begin{abstract}

We consider several transformation groups of a locally conformally
K\"ahler manifold and discuss their inter-relations. Among other
results, we prove that all conformal vector fields on a compact
Vaisman manifold which is neither locally conformally hyperk\"ahler
nor a diagonal Hopf manifold are Killing, holomorphic and that all
affine vector fields with respect to the minimal Weyl connection of a locally
conformally K\"ahler manifold which is neither Weyl-reducible nor
locally conformally hyperk\"ahler are holomorphic and conformal. 

\vs

\noindent


\end{abstract}

\maketitle

\section{Introduction}
We briefly present the necessary background for locally conformally
K\"ahler (in short, LCK) geometry. 

In the sequel, $J$ will denote an integrable
complex structure on a connected, smooth
manifold $M^{2n}$ of complex dimension $n\ge 2$. For a Hermitian metric $g$ on $(M,J)$, we denote
by $\nabla^g$  its Levi-Civita connection and by
$\omega$ its fundamental two-form 
$\omega(X,Y):=g(JX,Y)$.

Traditionally, a LCK metric $g$ on the complex manifold $(M,J)$ is
defined by the following ``integrability'' condition satisfied by its
fundamental two-form: 
\beq\label{lee0}d\omega=\theta\wedge\omega,\quad d\theta=0.\eeq 
The closed one-form $\theta$ is called the \emph{Lee form}. Note that,
on manifolds of complex dimension at least $3$, the first equation
implies the closedness of $\theta$, the second condition being
relevant only on complex surfaces. 

Any other Hermitian metric $e^fg$ which is conformal with $g$ is a LCK
metric too, its Lee form  being $\theta+df$. Hence, a LCK metric
determines a $1$-cocycle associated merely to the conformal class
$c=[g]$.  

Clearly, on local open sets on which $\theta$ is exact, $g$ is
conformal with some local K\"ahler metrics which, on overlaps, are
homothetic. The above $1$-cocycle is associated to this system of
scale factors.  

Note that on the universal cover of $M$, the pull-back of $\theta$ is
exact (the pull-backs of the local K\"ahler metrics glue together to a
global K\"ahler metric) and the fundamental group of $M$ acts by
biholomorphic homotheties with respect to this K\"ahler metric. This
is usually taken for definition of a LCK manifold. 

Returning on $M$, the Levi-Civita connection $D^U$ of a local K\"ahler
metric in the conformal class of $g$ on some open set $U$ is related
to the Levi-Civita connection of $g$ by the formula: 
\beq\label{lee}D^U=\nabla^g-\frac 12 (\theta\otimes
\id+\id\otimes\theta -g\otimes \theta^\sharp).\eeq 
It is obvious that, in fact, these connections do not depend on the
particular local K\"ahler metric, \emph{i.e.} they glue together to a
global connection, here denoted $D$, which has the following two
properties: 
$$DJ=0, \quad Dg=\theta\otimes g.$$
The second equation implies that $D$ preserves the conformal class:
$Dc=0$. Being torsion free (as, locally, any Levi-Civita connection),
it is a Weyl connection. Notice that the Lee form of the Weyl
connection $D$ with respect to $g$ defined by (\ref{lee}) coincides
with the Lee form of the complex structure $J$ with respect to $g$
defined by (\ref{lee0}). 

\begin{ede} A \emph{Hermitian-Weyl structure} on $M^{2n}$ is a triple
  $(c,J,D)$ where $c$ is a conformal structure, $J$ is a complex
  structure compatible with $c$ and $D$ is a  Weyl connection such
  that $DJ=0$ and $Dc=0$. 
\end{ede}

From the above, every LCK structure defines a Hermitian-Weyl structure
on $M$, and conversely, every metric in the conformal class of a
Hermitian-Weyl structure is LCK, provided $n\ge 3$. Notice that the
Weyl connection $D$ is uniquely defined by $J$ (see {\em e.g.} \cite[
Lemma 5.1]{bm}). We call $D$ the {\em minimal} Weyl connection.

In general, there is no way to choose a ``canonical" metric in the
conformal class of a Hermitian-Weyl manifold $(M,c,J,D)$, unless $M$
is compact, where there exists a unique, up to homothety, metric $g_0$
in $c$ such that the Lee form $\theta_0$ of $D$ with respect to $g_0$
is co-closed (and hence harmonic), see \cite[p.502]{gau}. We call
$g_0$ the \emph{the Gauduchon metric}. 

A particular class of LCK manifolds are the Vaisman manifolds. These
are defined by the condition $\nabla^g\theta=0$. By the uniqueness up
to homotheties, a Vaisman metric is, necessarily, the Gauduchon metric
in its conformal class. The prototype of Vaisman manifolds are the
Hopf manifolds $\CM^n\setminus\{0\}/\ZM$, with $\ZM$ generated by a
semi-simple endomorphism (see \cite{ov} for the structure of compact
Vaisman manifolds). On the other hand, there exist examples of compact
LCK manifolds which do not admit any Vaisman metric: such are the
Inoue surfaces and the non-diagonal Hopf surfaces, see \cite{be}. 

The hypercomplex  version of LCK geometry is straightforward. One
starts with a hypercomplex manifold $(M,I,J,K)$ endowed with a
conformal class $[c]$ and with a Weyl connection $D$ satisfying: 
$$Dc=0, \quad DI=DJ=DK=0.$$
One obtains the notion of hyperhermitian-Weyl structure. It can be
seen that for any $g\in c$, the Lee forms associated to the three complex structures 
coincide. As above, the Lee form is closed in real dimension
at least $8$, and hence, with the exception of dimension $4$, \emph{a
  hyperhermitian-Weyl manifold is locally conformally hyperk\"ahler
  (LCHK)}. 

The local K\"ahler metrics of a LCHK manifold are now hyperk\"ahler,
hence Ricci-flat. It follows that LCHK manifolds are necessarily
Einstein-Weyl. But on compact Einstein-Weyl manifolds, the Gauduchon
metric is parallel, see \cite{gau95}. Hence, a compact LCHK manifold
bears three nested Vaisman structures. 

For examples and properties of LCK and LCHK manifolds, we refer to
\cite{do}, to more recent papers by Ornea and Verbitsky and to the
references therein. 

\begin{ere} {\rm 
 In the whole paper, we tacitly assume that the manifolds we consider
 are not globally conformally K\"ahler, \emph{i.e.} the Lee form is
 never  exact. This is especially important on compact manifolds,
 where LCK and K\"ahler structures impose completely different
 topologies.  } 
\end{ere}

We now consider the following transformation groups on a
Hermitian-Weyl manifold $(M,c,J,D)$:  
\begin{itemize}
\item $\mathrm{Aff}(M,D)$, the group of affine transformations, i.e. preserving the Weyl connection. 
\item $\mathrm{H}(M,J)$, the group of biholomorphisms with respect to $J$.
\item $\mathrm{Conf}(M,c)$, the group of conformal transformations. 
\item $\mathrm{Aut}(M)$:= $\mathrm{Hol}(M,J)\cap \mathrm{Conf}(M,c)$,
  the group of automorphisms.
\end{itemize}

It is well known that these are Lie groups. Their Lie algebras will be
denoted, respectively by: $\mathfrak{aff}(M,D)$, $\mathfrak{h}(M,J)$,
$\mathfrak{conf}(M,c)$, $\mathfrak{aut}(M)$.  

On Vaisman manifolds, the Lee field $\theta^\sharp$ is Killing and
real-analytic but on non-Vaisman LCK manifolds, almost no information
about these groups is available. 
 It is the purpose of this note to clarify some relations among the
 above groups. Essentially, we prove that: 
\begin{itemize}
 \item $\mathfrak{aff}(M,D)=\mathfrak{aut}(M)$, provided that
   $\mathrm{Hol}_0(D)$ is irreducible and $M$ is not LCHK (Corollary
   \ref{af_field}). 
\item $\mathfrak{conf}(M,c)=\mathfrak{aut}(M)$ on compact Vaisman
  manifolds which are neither LCHK nor diagonal Hopf manifolds
  (Theorem \ref{ckh}).  
\end{itemize}

Note that on compact LCHK manifolds and on Hopf manifolds there exist
examples of affine transformations which are not holomorphic, see
Remark \ref{ir_vai} (ii) below. 

We believe that the second equality holds on all compact LCK manifolds
which are neither LCHK nor diagonal Hopf manifolds. By Lemma
\ref{kill_gau} below, this amounts to prove that every conformal
vector field is Killing for the Gauduchon metric, which is the LCK
counterpart of the result saying that every conformal vector field on
a compact K\"ahler manifold is Killing (\cite[\S 90]{li}).

\section{Affine vector fields on LCK manifolds}

Our first result is the LCK analogue of \cite[\S 54]{li},  (see also
\cite{kn}). 
 
\begin{elem}  
Let $(M,c,J,D)$ be a LCK manifold which is not locally conformally
hyperk\"ahler and such that $\mathrm{Hol}_0(D)$ is irreducible. Then
any $f\in  \mathrm{Aff}(M,D)$ is holomorphic or anti-holomorphic.  
\end{elem}
\begin{proof}
Let $f\in  \mathrm{Aff}(M,D)$ and let $J'_x:=(d_xf)^{-1}\circ
J_{f(x)}\circ (d_xf)$ denote the image by $f$ of the complex structure
$J$. Then ${J'}^2=-\id$ and, as $J$ and $df$ commute with the parallel
transport induced by $D$, we easily derive that $J'$ is $D$-parallel
too.  

Consider the decomposition of $JJ'$ into symmetric and skew-symmetric
parts: $JJ'=S+A$, where 
$$
\begin{cases}
S:=\frac 12 (JJ'+J'J),\\
A:=\frac 12 (JJ'-J'J).
\end{cases}
$$
Since $S$ is $D$-parallel, its eigenvalues are constant and the
corresponding eigenbundles are $D$-parallel. The irreducibility
assumption implies $S=k\id$ for some $k\in\RM$. Similarly, the
$D$-parallel symmetric endomorphism $A^2$ has to be a multiple of the
identity: $A^2=p\id$.  

If $A$ were non-zero, $A(X) \ne 0$ for some $X\in TM$, so
$$0>-c(AX,AX)=c(A^2X,X)=pc(X,X),$$
whence $p<0$. The endomorphism $K:=A/\sqrt{-p}$ is then a $D$-parallel
complex structure and satisfies $KJ=-JK$, so $(J,K)$ defines a LCHK
structure on $(M,c)$, which is forbidden by the hypothesis. Thus $A=0$
and $JJ'=k\id$, so $J'=-kJ$. Using ${J'} ^2=-\id$ we get $k=\pm 1$, so
$J'=\pm J$, which just means that $f$ is holomorphic or
anti-holomorphic.

\end{proof}

In a similar manner one can prove the following 
\begin{elem}  
Let $(M, c, J, D)$  be a LCK manifold such that $\mathrm{Hol}_0(D)$ is
irreducible. Then any $f\in  \mathrm{Aff}(M,D)$ is conformal.  
\end{elem}

\begin{proof}
Since $f$ is affine, the pull-back by $f$ of the conformal structure
$c$ is a $D$-parallel conformal structure $c'$. The symmetric
endomorphism $B$ of $TM$ defined by $c'(X,Y)=c(BX,Y)$ is
$D$-parallel. The irreducibility of $\mathrm{Hol}_0(D)$ shows as
before that $B$ is a multiple of the identity. 
\end{proof}

\begin{ecor}\label{af_field}
Let $(M,c,J,D)$ be a LCK manifold which is not LCHK and such that
$\mathrm{Hol}_0(D)$ is irreducible. Then any infinitesimal affine
transformation of the Weyl connection is an infinitesimal
automorphism, {\em i.e.} $\mathfrak{aff}(M,D)=\mathfrak{aut}(M)$.  
\end{ecor}

\begin{ere}\label{ir_vai}
{\rm $(i)$ We do not know whether the assumption of irreducibility of
  the Weyl connection can be relaxed or not, at least on compact
  $M$. But we  can show that \emph{a compact Vaisman manifold, which
    is not a (diagonal) Hopf manifold, is Weyl-irreducible}. Indeed,
  by the structure theorem in \cite{ov}, $M$ is a mapping torus of a
  Sasakian isometry and its universal cover is a K\"ahlerian cone over
  the compact, hence complete, Sasakian fibre. Then, as $D$ is,
  locally, the Levi-Civita connection of the local K\"ahler metrics,
  if $D$ is reducible, the K\"ahler metric of the covering cone is
  reducible. Now, by Proposition 3.1 in \cite{gal}, a cone over a
  complete manifold is reducible if and only if it is flat. But a flat
  cone is the cone over a sphere, hence $M$ is a Hopf manifold.  

$(ii)$ Let $W$ be a 3-Sasakian manifold. Then $M:=W\times S^1$ is LCHK
(for $W=S^{2n-1}$, $M$ is a Hopf manifold) and each of the three Killing
fields on $W$ which generate the $\SU(2)$ action, induce Killing,
hence affine with respect to the Weyl connection, fields on $M$ and
each of them is holomorphic only with respect to one of the three complex
structures (see \emph{e.g.} \cite[Chapter 11]{do}).} 

\end{ere}

We apply the above to prove:

\begin{elem}\label{kill_gau}
Let $(M,c,J, D)$ be a compact LCK manifold which is not LCHK and such
that $\mathrm{Hol}_0(D)$ is irreducible. Then every Killing vector field
$\xi$ of the Gauduchon metric is a holomorphic vector field. 
\end{elem}
\begin{proof}
According to Corollary \ref{af_field}, it is enough to prove that $\xi$ is
affine. As the Weyl connection is uniquely defined by the Gauduchon
metric and its Lee form $\theta_0$, it will be enough to show that  
\begin{equation}\label{unu}
\Lie_\xi\theta_0=0.
\end{equation} 
But, as $\xi$ is $g_0$-Killing, the codifferential $\delta$ of $g_0$
commutes with $\Lie_\xi$, hence
$\delta(\Lie_\xi\theta_0)=\Lie_\xi(\delta\theta_0)=0$, because
$\theta_0$ is co-exact. Now, $\theta_0$ is closed, so
$d(\Lie_\xi\theta_0)=0$, {\em i.e.} $\Lie_\xi\theta_0$ is harmonic. On
the other hand, by Cartan's formula, $\Lie_\xi\theta_0=d(\xi\es
\theta_0)$. On a compact Riemannian manifold, a one-form which is both
exact and harmonic necessarily vanishes, so \eqref{unu} is proved. 
\end{proof}

\section{Conformal vector fields on Vaisman manifolds}

As mentioned in the introduction, every conformal vector field on a
compact K\"ahler manifold is Killing (\cite[\S 90]{li}). We consider
the LCK analogue of this statement: {\em Is every conformal vector
  field on a compact LCK manifold Killing with respect to the
  Gauduchon metric?}

As no sphere $S^{2n}$ can bear an LCK metric for $n\ge 2$ 
($S^{2n}$ being simply connected,
such a structure would be automatically K\"ahler), by Obata's theorem
we derive that \emph{any conformal vector field on a compact LCK
  manifold is Killing for some metric in the given conformal class,
  but not necessarily the Gauduchon metric}. In fact, an argument
entirely similar to the above, shows that a conformal vector field is
Killing with respect to the Gauduchon metric if and only if it
preserves the Lee form of the metric with respect to which it is
Killing. 

The answer to the above question is not known to us in
general. Nevertheless, we can show that  
it holds on compact Vaisman manifolds, on which every conformal field
is Killing with respect to the Gauduchon metric. This, actually,
follows from a more general statement: 

\begin{ath}\label{mt}
Let $\f$ be an isometry of a compact Riemannian manifold $(W,h)$ and
let $(M,g):=(W,h)\times\RM/_{\{(x,t)\sim(\f(x),t+1)\}}$ be the
mapping torus of $\f$. Then every conformal vector field on $(M,g)$ is
Killing.  
\end{ath}
\begin{proof}
In \cite{ms} it is shown that every twistor form on a Riemannian
product is defined by Killing forms on the factors. We adapt the
argument there to the situation of mapping tori.  

Every conformal vector field on $(M,g)$ induces a conformal vector
field denoted $\xi$ on the covering space $W\times\RM$ of $M$, endowed
with the product metric $\tilde g:=h+dt ^2$. We write $\xi$ as 
$$\xi=a\dt+\eta,$$ 
where $\dt=\partial/\partial t$, $\eta$ is tangent to $W$ and $a$
is a function. Both $a$ and $\eta$ can be viewed as objects on $W$
indexed upon the parameter $t$ on $\RM$. Since $\xi$ is invariant under
the isometry $(x,t)\mapsto(\f(x),t+1)$ of $W\times \RM$, we get in
particular $a(w,t)=a(\f(w), t+1),$
and thus, if we write $a_t(w)$ for $a(w,t)$,
\beq\label{ab}a_{t}=\f^*a_{t+1}, \quad \dot a_{t}=\f^*\dot a_{t+1},\qquad
\forall\ (w,t)\in W\times\RM.\eeq
Now, $\xi$ being conformal, there exists a function $f$ on $W\times\RM$
such that $\Lie_\xi \tilde
g=f\tilde g$. This can be written 
\begin{equation}\label{doi}
\tilde g(\nabla_A\xi,B)+\tilde g(A,\nabla_B\xi)=f\tilde g(A,B)
\end{equation}
for every tangent vectors to $W\times\RM$, $A$ and $B$.
Taking $A=B=\dt$ in (\ref{doi}) gives
$$f=2\tilde g(\nabla_\dt\xi,\dt),$$
hence, by the parallelism of $\dt$, one has
\beq\label{fa}f=2\dt(a)=2\dot a.\eeq
Applying the same equation \eqref{doi} on pairs $(\dt,X)$, $(Y,Z)$, with
$X,Y,Z$ tangent to $W$ and independent on $t$, we obtain, respectively: 
\begin{equation*}
\nabla_{\dt}\eta+da=0,
\end{equation*}
and
\begin{equation*}
h(\nabla_Y\eta,Z)+h(Y,\nabla_Z\eta)=fh(Y,Z).
\end{equation*}
Notice that here $da$ denotes the exterior derivative of $a$ on $W$
and should be understood as a family of 1-forms on $W$ indexed on $t$,
identified by $h$ with a family of vector fields on $W$.
Taking (\ref{fa}) into account, the two equations above become: 
\begin{equation}\label{tr}
\begin{cases}
\dot \eta=-da,\\
h(\nabla^W_Y\eta,Z)+h(Y,\nabla^W_Z\eta)=2 \dot ah(Y,Z) 
\end{cases}
\end{equation}
As the vector fields $Y,Z$ are independent on $t$, we can take the
derivative with respect to $t$ in the second equation of \eqref{tr}: 
$$h(\nabla^W_Y\dot \eta,Z)+h(Y,\nabla^W_Z \dot \eta)=2\ddot a h(Y,Z).$$
Letting here $Y=Z=E_i$, where $\{E_i\}$ is a local orthonormal basis
on $W$, we obtain  
$$\delta^W \dot \eta=-m\ddot a, \qquad (m=\dim W).$$
Using the first equation in \eqref{tr}, we finally obtain:
\beq\label{del}\delta^W d a=m\ddot a.\eeq
We use this equation to show that $a$ is constant. To this end, we
integrate $\Vert da\Vert^2$ on $W$: 
$$\int_Wh(da,da)=\int_W(\delta^W d a) a=m\int_Wa\ddot a=m\int_W(a\dot
a)'-m\int_W(\dot a)^2.$$ 
If we let $b:=ma\dot a$, then, from the above equation we have 
\beq\label{ab1}\int_W\dot b=\int_Wh(da,da)+m\int_W(\dot a)^2\ge0.\eeq 
On the other hand, taking $t=0$ in (\ref{ab}) yields
$$\int_Wb_{0}=\int_W\f^*b_1=\int_Wb_1,$$
because $\f$ is a diffeomorphism. We can therefore compute
$$\int_{W\times [0,1]}\dot b=\int_0^1\bigg(\int_W\dot
b\bigg)dt=\int_Wb_{1}-\int_Wb_0=0.$$ 
By (\ref{ab1}), $\dot a=0$ and $da=0$. As $M$ is connected, $a$ is
constant. It follows that $f=0$ and thus $\xi$ is Killing. 

\end{proof}

We apply the previous result to compact Vaisman manifolds which, by
the structure theorem in \cite{ov}, are mapping tori with compact,
Sasakian fibre. As a direct consequence of Remark \ref{ir_vai}(i),
Lemma \ref{kill_gau} and Theorem \ref{mt} we get: 

\begin{ath}\label{ckh}
Any conformal vector field on a compact Vaisman manifold which is
neither LCHK nor a diagonal Hopf manifold is Killing and holomorphic. 
\end{ath}


\begin{thebibliography}{22}

{
\bibitem{be} F. Belgun, {\it On the metric structure of non-K\"ahler
    complex surfaces},  Math. Ann.  {\bf 317}  (2000), 1--40. 

\bibitem{bm} F. Belgun, A. Moroianu, {\it Weyl-parallel Forms and conformal Products}, arXiv:0901.3647.

\bibitem{do} S. Dragomir, L. Ornea, {\it Locally conformal K\"ahler
    geometry}, Progress in Math. {\bf 155}, Birkh\"auser, 1998. 

\bibitem{gal} S. Gallot, {\it Equations diff\'erentielles
    caract\'eristiques de la sph\`ere}, Ann. Sci. Ec. Norm. Sup. {\bf
    12}, 235--267 (1979). 

\bibitem{gau} P. Gauduchon, {\it La 1-forme de torsion d'une
    vari\'et\'e hermitienne compacte}, Math. Ann. {\bf 267}, 495--518
  (1984). 

\bibitem{gau95} P. Gauduchon, {\it Structures de Weyl-Einstein,
    espaces de twisteurs et vari\'et\'es de type $S\sp 1\times S\sp
    3$},  J. Reine Angew. Math.  {\bf 469}  (1995), 1--50. 

\bibitem{kn} S. Kobayashi, K. Nomizu, {\it On automorphisms of a
    K\"ahlerian structure},  Nagoya Math. J.  {\bf 11}, 115--124
  (1957). 

\bibitem{li} A. Lichn\'erowicz, {\it G\'eom\'etrie des groupes de
    transformations}, Dunod, Paris, 1958. 

\bibitem{ms} A. Moroianu, U. Semmelmann, {\it Twistor Forms on
    Riemannian Products}, J. Geom. Phys. {\bf 58}, 1343--1345 (2008). 

\bibitem{ov} L. Ornea, M. Verbitsky, {\it Structure theorem for
    compact Vaisman manifolds}, Math. Res. Lett., {\bf 10}, 799--805
  (2003).  

}

\end{thebibliography}
\end{document}